\documentclass[a4paper,oneside,11pt, reqno]{amsart}

\usepackage{mathrsfs}
\usepackage{amsfonts,amssymb,amsmath,amsthm}
\usepackage{url}
\usepackage{enumerate}
\usepackage{enumitem}
\usepackage{amssymb}
\usepackage{bbm}
\usepackage{float}
\usepackage{pstricks}
\usepackage{amscd}
\usepackage{amsmath}
\usepackage{amsxtra}
\usepackage[T1]{fontenc}
\usepackage[utf8]{inputenc}
\usepackage{lipsum}
\usepackage{tikz}
\usetikzlibrary{arrows}
\usetikzlibrary{calc}
\usepackage{hyperref}

\theoremstyle{plain}
\newtheorem{theorem}{Theorem}[section]

\newtheorem{lemma}[theorem]{Lemma}
\newtheorem{proposition}[theorem]{Proposition}

\theoremstyle{definition}

\theoremstyle{remark}
\newtheorem{remark}[theorem]{Remark}

\newtheorem{case-new}{Case}

\numberwithin{equation}{section}

\usepackage{etoolbox}

\makeatletter
\@namedef{subjclassname@2020}{%
  \textup{2020} Mathematics Subject Classification}
\makeatother

\usepackage{mathrsfs}
\usepackage{epsfig}
\usepackage[active]{srcltx}

\newcommand{\ncom}{\newcommand}
\ncom{\ul}{\underline}
\ncom{\ol}{\overline}
\ncom{\bq}{\begin{equation}}
\ncom{\eq}{\end{equation}}
\ncom{\beqn}{\begin{eqnarray*}}
\ncom{\eeqn}{\end{eqnarray*}}
\ncom{\beq}{\begin{eqnarray}}
\ncom{\eeq}{\end{eqnarray}}
\ncom{\nno}{\nonumber}
\ncom{\rar}{\rightarrow}
\ncom{\Rar}{\Rightarrow}
\ncom{\noin}{\noindent}
\ncom{\bc}{\begin{centre}}
\ncom{\ec}{\end{centre}}
\ncom{\sz}{\scriptsize}
\ncom{\rf}{\ref}
\ncom{\sgm}{\sigma}
\ncom{\Sgm}{\Sigma}
\ncom{\dt}{\delta}
\ncom{\Dt}{Delta}
\ncom{\lmd}{\lambda}
\ncom{\Lmd}{\Lambda}
\ncom{\eps}{\epsilon}
\ncom{\pcc}{\stackrel{P}{>}}
\ncom{\dist}{{\rm\,dist}}
\ncom{\sspan}{{\rm\,span}}
\ncom{\im}{{\rm Im\,}}
\ncom{\sgn}{{\rm sgn\,}}
\ncom{\ba}{\begin{array}}
\ncom{\ea}{\end{array}}
\ncom{\eop}{\hfill{{\rule{2.5mm}{2.5mm}}}}
\ncom{\eoe}{\hfill{{\rule{1.5mm}{1.5mm}}}}
\ncom{\eof}{\hfill{{\rule{1.5mm}{1.5mm}}}}
\ncom{\hone}{\mbox{\hspace{1em}}}
\ncom{\htwo}{\mbox{\hspace{2em}}}
\ncom{\hthree}{\mbox{\hspace{3em}}}
\ncom{\hfour}{\mbox{\hspace{4em}}}
\ncom{\hsev}{\mbox{\hspace{7em}}}
\ncom{\vone}{\vskip 2ex}
\ncom{\vtwo}{\vskip 4ex}
\ncom{\vonee}{\vskip 1.5ex}
\ncom{\vthree}{\vskip 6ex}
\ncom{\vfour}{\vspace*{8ex}}
\ncom{\norm}{\|\;\;\|}
\ncom{\integ}[4]{\int_{#1}^{#2}\,{#3}\,d{#4}}
\ncom{\inp}[2]{\langle{#1},\,{#2} \rangle}
\ncom{\Inp}[2]{\Langle{#1},\,{#2} \Langle}
\ncom{\vspan}[1]{{{\rm\,span}\#1 \}}}
\ncom{\dm}[1]{\displaystyle {#1}}

\usepackage{ulem}

\keywords{Invariant subspaces, Brownian shift, Inner functions, Hardy
spaces, Brownian unitary, 2-isometry, 3-isometry}

\subjclass[2020]{30H10, 30J05, 47A15, 47A20, 60J65}

\begin{document}

\title[$3\times 3$ BROWNIAN
Shift]{Invariant Subspaces and the $C_{00}$-Property of $3$-Brownian
Shifts}

\author[R. Nailwal]{Rajkamal Nailwal}
\address{Rajkamal Nailwal, Institute of Mathematics, Physics and Mechanics, Ljubljana, Slovenia.}
\email{rajkamal.nailwal@imfm.si, raj1994nailwal@gmail.com}

\begin{abstract}
In this paper, we introduce a $3$-Brownian shift $T_{\sigma, \theta}$ on the Hilbert space
$H^2(\mathbb D^2)\oplus H^2(\mathbb D)\oplus \mathbb C,$ 
which is a natural extension of the classical Brownian shift $B_{\sigma, \theta}$ on  $H^2(\mathbb D)\oplus \mathbb C$. This is motivated by Brownian extensions in the context of 3-isometries recently developed by A. Crăciunescu and L. Suciu. We investigate the problem of unitary equivalence for  $3$-Brownian shifts on invariant subspaces of the type $\mathcal M_0 \oplus \mathcal M_1,$ where $\mathcal M_0 \subseteq H^2(\mathbb D^2)$  and $\mathcal M_1 \subseteq H^2(\mathbb D)\oplus \mathbb C.$ Here, $\mathcal M_1$ turns out to be an invariant subspace of the respective Brownian shift $B_{\sigma, \theta}$. We also study the asymptotic behaviour of the normalized $3$-Brownian shifts. This work is motivated by Richter \cite{R88} and very recently by work on Brownian shift on $H^2(\mathbb D)\oplus \mathbb C$ in \cite{DDS2025}.
\end{abstract}
\maketitle
\section{Introduction}

Let $\mathbb C$, $\mathbb D$, and $\mathbb T$ denote the complex plane, the open unit disk, and the unit circle, respectively. 
Let $\mathcal H$ be a complex separable Hilbert space, and let $\mathcal{B}(\mathcal H)$ denote the $C^*$-algebra of bounded linear operators on $\mathcal H$. 
For $V \in \mathcal{B}(\mathcal H)$, the subspaces $\mathcal{N}(V)$ and $\mathcal{R}(V)$ denote the kernel (null space) and the range of $V$, respectively. A closed subspace $M \subseteq \mathcal H$ is said to be {\it invariant} under $T \in \mathcal B(\mathcal H)$ if $T(M)\subseteq M.$

In this paper, we consider the Hardy spaces of the unit disk and the unit bidisk, denoted by $H^2(\mathbb D)$ and $H^2(\mathbb {D} ^2)$, respectively. The Hardy space of the unit disk is defined by
\[
H^2(\mathbb D)
=
\left\{
f(z)=\sum_{n\geq 0} a_n z^n :
\sum_{n\geq 0} |a_n|^2 < \infty
\right\},
\]
with inner product
\[
\langle f,g\rangle
=
\sum_{n\geq 0} a_n \overline{b_n},
\]
where $f(z)=\sum_{n\geq 0}  a_n z^n$ and $g(z)=\sum_{n\geq 0}  b_n z^n$.

The Hardy space of the unit bidisk is given by
\[
H^2(\mathbb D^2)
=
\left\{
F(z_1,z_2)
=
\sum_{m,n \geq 0} a_{m,n} z_1^m z_2^n :
\sum_{m,n \geq 0} |a_{m,n}|^2 < \infty
\right\},
\]
equipped with the corresponding $\ell^2$ coefficient norm.

We will also need vector-valued Hardy spaces. For a Hilbert space $E$, let 
$H^2(\mathbb D;E)$ denote the space of $E$-valued analytic functions on $\mathbb D$ 
with square-summable Taylor coefficients.

Given the Hilbert spaces $E$ and $F,$ we denote by 
$H^\infty(\mathbb D;\mathcal B(E,F))$ the Banach space of bounded analytic 
$\mathcal B(E,F)$-valued functions on $\mathbb D$, where $\mathcal B(E,F)$ denotes 
the space of bounded linear operators from $E$ to $F$. In the scalar case, when
$E=F=\mathbb C$, we simply write $H^\infty(\mathbb D)$.

For each $\varphi \in H^\infty(\mathbb D;\mathcal B(E,F)),$ the multiplication operator $M_\varphi $ is defined by
\[
M_\varphi : H^2(\mathbb D;E) \to H^2(\mathbb D;F), 
\qquad (M_\varphi f)(z)=\varphi(z)f(z).
\]

A function $\varphi \in H^\infty(\mathbb D;\mathcal B(E,F))$ is called 
{\it inner} if $M_\varphi$ is an isometry. Equivalently, the radial boundary values 
satisfy that $\varphi(e^{it})$ is an isometry from $E$ to $F$ for almost every 
$t\in[0,2\pi)$.

 For such an inner function $\varphi$, the associated model space is given by
\[
K_\varphi := H^2(\mathbb D;F)\ominus \varphi H^2(\mathbb D;E).
\]

 Classically, a Brownian shift in the Hilbert space $H^2(\mathbb D)\oplus\mathbb C$ is defined by
\[
B_{\sigma,e^{i\theta}}
=
\begin{bmatrix}
S & \sigma(1\otimes1)\\
0 & e^{i\theta}
\end{bmatrix},
\]
where $S$ denotes the unilateral shift on $H^2(\mathbb D)$, $\sigma>0$ is a covariance parameter, and $\theta\in[0,2\pi)$ (\cite[Definition~5.5]{AS95}).  The operator \( 1 \otimes 1 : \mathbb{C} \to H^2(\mathbb D) \) is defined by assigning to each scalar \( \alpha \in \mathbb{C} \) the constant function on \( \mathbb{D} \) with the value \( \alpha \), i.e.,
\[
(1 \otimes 1)(\alpha)(z) = \alpha, \quad \text{for all } z \in \mathbb{D}.
\] The Brownian shift $B_{\sigma, \theta}$   is a special case of Brownian unitaries, also referred to as  B-unitary (\cite[Definition~1.1]{CJJS24I}), which plays a special role when studying $2$-isometries (\cite[Section~5]{AS95}, also see \cite{JJS2026, S23}). 

Motivated by Brownian unitaries in the context of $2$-isometries, Crăciunescu and  Suciu recently developed the following
$3$-Brownian unitaries in the context of $3$-isometries \cite{CS2024}. 
Given a Hilbert space $\mathcal{H}$  with an orthogonal decomposition
\[
\mathcal{H} = \mathcal{H}_0 \oplus \mathcal{H}_1 \oplus \mathcal{H}_2,
\]
an operator $B \in \mathcal{B}(\mathcal{H})$ is called a {\it 3-Brownian unitary} 
if, with respect to this decomposition, it has the block matrix representation
\[
B =
\begin{bmatrix}
V_0 & \sigma E_0 & 0 \\
0   & V_1        & \sigma E_1 \\
0   & 0          & U
\end{bmatrix},
\]
where $V_j, E_j$ are isometries satisfying
\[
\mathcal{N}(V_j^*) = \mathcal{R}(E_j), \qquad j=0,1,
\]
$U$ is unitary, and $\sigma > 0$.

The Brownian shift $B_{\sigma, \theta}$ first appears as part of studying time shift operators (\cite[Section~5]{AS95}) and has become an important member of the larger class called $B$-operator which turns out to be a very useful object in understanding several properties in operator theory \cite{CJJS24}. Motivated by this, we introduce a $3$-Brownian shift on $H^2(\mathbb D^2)\oplus H^2(\mathbb D)\oplus \mathbb C.$  A {\it $3$-Brownian shift } is a  3-Brownian unitary on $H^2(\mathbb D^2)\oplus H^2(\mathbb D)\oplus \mathbb C$ of the following form

\[
T_{\sigma,\theta}
=
\begin{bmatrix}
M_{z_1} & \sigma J & 0\\
0 & S & \sigma(1\otimes1)\\
0 & 0 & e^{i\theta}
\end{bmatrix},
\]
where $M_{z_1}$ denotes multiplication by $z_1$ on $H^2(\mathbb D^2)$, $S$ is the unilateral shift on $H^2(\mathbb D)$, $J:H^2(\mathbb D)\to H^2(\mathbb D^2)$ is the canonical isometric embedding defined by $(Jf)(z_1,z_2)=f(z_2),$ $\sigma>0$  and $\theta \in [0,2\pi)$. The operator $T_{\sigma,\theta}$ can be viewed as the Brownian shift analogue in the setting of $3$-isometries.

Let $\mathcal M_1$ and $\mathcal M_2$ be two invariant subspaces of $T_{\sigma_1,\theta_1}$ and 
$T_{\sigma_2,\theta_2}$, respectively. We say that $\mathcal M_1$ and $\mathcal M_2$ are 
{\it unitarily equivalent}, and write $T_{\sigma_1,\theta_1}\big|_{\mathcal M_1} \cong  T_{\sigma_2,\theta_2}\big|_{\mathcal M_2}$, if there exists 
a unitary operator $U : \mathcal M_1 \to \mathcal M_2$ such that
\[
U\, T_{\sigma_1,\theta_1}\big|_{\mathcal M_1}
=
T_{\sigma_2,\theta_2}\big|_{\mathcal M_2}\, U .
\]

Motivated by the results of Richter \cite{R88} and more recently by the work of Das et al. \cite{DDS2025}, we investigate the problem of unitary equivalence of $3$-Brownian shifts when restricted to certain lifted type invariant subspaces (cf. \cite{D11, DS08}). In the classical case of the Hardy space of the unit disk, the famous Beurling theorem yields that all nonzero invariant subspaces of $H^2(\mathbb D)$ are unitarily equivalent. In the polydisk case, this was investigated in \cite{ACD86} (also see \cite{R91}).

A fundamental structural feature of the operator $T_{\sigma,\theta}$ is that its
invariant subspaces necessarily project to invariant subspaces of the lower
$2\times2$ Brownian shift. More precisely, if
\[
P:\mathcal H \to H^2(\mathbb D)\oplus\mathbb C,\qquad P(F,f,\alpha)=(f,\alpha),
\]
denotes the canonical projection and
\[
B_{\sigma,e^{i\theta}}(f,\alpha)=(zf+\sigma\alpha,\ e^{i\theta}\alpha)
\]
is the lower Brownian shift, then one has the intertwining relation
\[
PT_{\sigma,\theta}=B_{\sigma,e^{i\theta}}P.
\]
As a consequence, for any $T_{\sigma,\theta}$-invariant subspace $\mathcal M\subset\mathcal H$,
the projected subspace $P(\mathcal M)$ is automatically invariant for $B_{\sigma,e^{i\theta}}$. This perspective justifies our focus on invariant subspaces of $T_{\sigma,\theta}$ whose
lower components are prescribed to be of certain types, namely Type~I or Type~II (to be defined below).

This observation has an important conceptual implication. The invariant subspace
structure of the classical Brownian shift is already completely understood, and its
invariant subspaces fall into well-known families (Type~I and Type~II in the sense of
Agler--Stankus \cite[p. 21]{AS95}). Since every invariant subspace of $T_{\sigma,\theta}$ must project onto
one of these known subspaces, there is no additional freedom at the level of the lower
block. 
In \cite{AS95}, Agler and Stankus present the following description of the invariant
subspaces for Brownian shifts. If $\mathscr M$ is a nonzero closed
subspace of $H^2(\mathbb D)\oplus\mathbb C$, then $\mathscr M$ is invariant under
$B_{\sigma,e^{i\theta}}$ if and only if it has one of the following forms:
\begin{equation}\label{T1}
    \mathscr M = \varphi H^2(\mathbb D) \oplus \{0\} \quad (\textit{Type I}),
\end{equation}
for some inner function $\varphi\in H^\infty(\mathbb D)$, or
\begin{equation} \label{Type-2}
\mathscr M = \mathbb C
\begin{bmatrix} g \\[2pt] 1 \end{bmatrix}
\oplus \bigl(\varphi H^2 (\mathbb D) \oplus \{0\}\bigr) \quad (\textit{Type II}),
\end{equation}
for some inner function $\varphi\in H^\infty(\mathbb D)$  such that $\varphi(e^{i\theta})$
exists, and
\begin{equation*}
g(z)=\sigma\cdot\frac{\overline{\varphi(e^{i\theta})}\,\varphi(z)-1}{z-e^{i\theta}},
\qquad z\in\mathbb D.
\end{equation*}
It is also worth noting that $g \perp \varphi H^2(\mathbb D)$
(see \cite[p.~23]{AS95}).

Note that in order to describe the invariant subspaces of $T_{\sigma,\theta}$ of the form
\[
\mathcal M =
\begin{bmatrix}
\mathcal M_0 \\
\mathcal M_1
\end{bmatrix},
\]
where $\mathcal M_1$ is a nontrivial invariant subspace of the corresponding Brownian shift $B_{\sigma,\theta}$, it suffices to determine $\mathcal M_0$. In the next section, we will deduce that $\mathcal M_0$ has the following two choices:
\begin{itemize}
    \item[(a)] If $\mathcal M_1$ is of Type I (see \eqref{T1})  for some inner function $\varphi \in H^{\infty}(\mathbb D)$, then 
$$\mathcal M_0
=
 \varphi_{z_2}\, H^2(\mathbb D^2) \;
\oplus\;
\Psi H^2(\mathbb D; \mathcal E_{\mathcal{M}})
$$
for some Hilbert space $\mathcal E_{\mathcal{M}},$ and  an operator-valued inner function
$
\Psi \in H^\infty\!\big(\mathbb D; \mathcal B(\mathcal E_{\mathcal{M}}, K_\varphi)\big).
$ Here $\varphi_{z_2}(z_1,z_2):=\varphi(z_2).$ We call the respective $\mathcal M$ as {\it lifted Type I invariant subspace} of $T_{\sigma, \theta}.$
    \item[(b)] If $\mathcal M_1$ is of Type II (see \eqref{Type-2}) for some inner function $\varphi \in H^{\infty}(\mathbb D)$, then 
    \[
\mathcal M_0
=
 \varphi_{z_2}\, H^2(\mathbb D^2) 
\;\oplus\;
g_{z_2}H^2_{z_1}(\mathbb D)
\;\oplus\;
\Psi\,H^2(\mathbb D;\mathcal E_{\mathcal{M}})
\]
for some Hilbert space $\mathcal E_{\mathcal{M}},$ and  an operator-valued inner function
$
\Psi \in H^\infty\!\big(\mathbb D; \mathcal B(\mathcal E_{\mathcal{M}}, K_\varphi)\big) $ satisfying  $
\Psi(0)\mathcal E_{\mathcal M} \subset g^\perp \subset K_\varphi
$.
Here $H^2_{z_1}(\mathbb D):=i_{z_1}(H^2(\mathbb D))$ is a closed subspace of $H^2(\mathbb D^2)$ where  $i_{z_1}:H^2(\mathbb D)\rightarrow H^2(\mathbb D^2)$ is an isometry defined by $(i_{z_1}f)(z_1,z_2):=f(z_1), f \in H^2(\mathbb D).$ The notation $g_{z_2}$ is defined by $g_{z_2}(z_1,z_2):=g(z_2)$  We call the respective $\mathcal M$ as {\it lifted Type II invariant subspace} of $T_{\sigma, \theta}.$ As $g \in K_{\varphi},$ the orthogonal complement $g^{\perp}$ is taken in $K_{\varphi}.$
\end{itemize}

Let us also consider the trivial case. Observe that if $\mathcal M_1=\{0\}$, then a necessary and sufficient condition for
$
\mathcal M_0 \, \oplus\, \{ 0\}
$
to be invariant under $T_{\sigma,\theta}$ is that $\mathcal M_0$ be an $M_{z_1}$-invariant subspace of $H^2(\mathbb D^2)$. Moreover, two such invariant subspaces are unitarily equivalent if and only if their first components are unitarily equivalent with respect to $M_{z_1}$. This situation can be completely described by the Beurling–Lax–Halmos theorem \cite[Theorem~2.1, p. 239]{FF90}.

We now state our main result of this paper.

\begin{theorem} \label{main}
Let $\theta_1,\theta_2\in[0,2\pi)$ and $\sigma_1,\sigma_2>0$. 
Let
\[
\mathcal M=
\begin{bmatrix}
\mathcal M_0\\
\mathcal M_1
\end{bmatrix},
\qquad
\mathcal N=
\begin{bmatrix}
\mathcal N_0\\
\mathcal N_1
\end{bmatrix}
\]
be nonzero closed invariant subspaces of lifted Type~I or Type~II for the $3$-Brownian shifts 
$T_{\sigma_1,\theta_1}$ and $T_{\sigma_2,\theta_2}$, respectively. Then
\[
T_{\sigma_1,\theta_1}\big|_{\mathcal M}
\cong
T_{\sigma_2,\theta_2}\big|_{\mathcal N}
\]
if and only if $\sigma_1=\sigma_2,\, \dim \mathcal E_{\mathcal M}=\dim \mathcal E_{\mathcal N}$ and one of the following holds:

\begin{enumerate}
\item[(i)] Both $\mathcal M$ and $\mathcal N$ are of lifted Type~I.
$
$
\medskip
\item[(ii)] Both $\mathcal M$ and $\mathcal N$ are of lifted Type~II, where
\[
\begin{aligned}
\mathcal M_0 &= \varphi_{z_2}^{(1)}H^2(\mathbb D^2)
\;\oplus\;
g_{z_2}^{(1)} H^2_{z_1}(\mathbb D)
\;\oplus\;
\Psi^{(1)} H^2(\mathbb D;\mathcal E_{\mathcal M}),\\
\mathcal N_0 &= \varphi_{z_2}^{(2)}H^2(\mathbb D^2)
\;\oplus\;
g_{z_2}^{(2)} H^2_{z_1}(\mathbb D)
\;\oplus\;
\Psi^{(2)} H^2(\mathbb D;\mathcal E_{\mathcal N}),
\end{aligned}
\]
and
\[
\begin{aligned}
\mathcal M_1 &= \mathbb C
\begin{bmatrix}
g^{(1)}\\
1
\end{bmatrix}
\oplus
\bigl(\varphi^{(1)}H^2(\mathbb D)\oplus\{0\}\bigr),\\
\mathcal N_1 &= \mathbb C
\begin{bmatrix}
g^{(2)}\\
1
\end{bmatrix}
\oplus
\bigl(\varphi^{(2)}H^2(\mathbb D)\oplus\{0\}\bigr),
\end{aligned}
\]
with
\[
g_j(z)
=
\sigma_j\cdot
\frac{\overline{\varphi^{(j)}(e^{i\theta_j})}\,\varphi^{(j)}(z)-1}{z-e^{i\theta_j}},
\qquad j=1,2,
\]
for some inner function $\varphi^{(j)}\in H^\infty(\mathbb D)$ and  for some operator-valued inner functions
\begin{equation}\label{A}
\Psi^{(1)} \in H^\infty\!\big(\mathbb D;\mathcal B(\mathcal E_{\mathcal M}, K_{\varphi^{(1)}})\big)\, \text{with} \,\,\,
\Psi(0)\mathcal E_{\mathcal M} \subset {g^{(1)}}^\perp \subset K_{\varphi^{(1)}}, \tag{A}
\end{equation}
\begin{equation}\label{B}
\Psi^{(2)} \in H^\infty\!\big(\mathbb D;\mathcal B(\mathcal E_{\mathcal M}, K_{\varphi^{(2)}})\big)\, \text{with }\,\, \,
\Psi(0)\mathcal E_{\mathcal N} \subset {g^{(2)}}^\perp \subset K_{\varphi^{(2)}}.\tag{B}
\end{equation}
such that
\[
\theta_1=\theta_2,\qquad \|g^{(1)}\|=\|g^{(2)}\|,
\]
\end{enumerate}
\end{theorem}
\begin{remark}
 The conditions $\eqref{A}$ and $\eqref{B}$ ensure that the decomposition of $\mathcal M_0$ and $\mathcal N_0$  in Theorem~\ref{main}(ii) is orthogonal, which is crucial for our result. For example, if we choose $\mathcal E_{\mathcal M}$ to be $K_{\varphi^{(1)}}$ and $\psi^{(1)}(z)=I$ where $ I$ is the identity operator, then the components $g_{z_2}^{(1)} H^2_{z_1}(\mathbb D)
$ and $
 H^2(\mathbb D;K_{\varphi^{(1)}})$ are not orthogonal as $g^{(1)}\in K_{\varphi^{(1)}}$.
\end{remark}

We now relate our analysis to the asymptotic property of Brownian shifts. 
Recall that if $T$ is a contraction on a Hilbert space $\mathcal H$, then $T$ is said to be {\it pure}, denoted $T\in C_{\cdot0}$, if
\[
\text{SOT-}\lim_{m\to\infty} T^{*m} = 0,
\]
where SOT denotes the strong operator topology. 
Moreover, $T$ is said to belong to the class $C_{00}$ if both $T$ and $T^*$ are pure.

Given $T \in B(\mathcal H)$ is said to be {\it power bounded} if the sequence  $\{\|T^n\|\}_{n=0}^{\infty}$ of real numbers is bounded.
It is known that the classical Brownian shift $B_{\sigma,\theta}$ is not power-bounded. 
In particular, it is not similar to a contraction. 
Nevertheless, after a suitable normalization, $B_{\sigma,\theta}$ belongs to the class $C_{00}$ \cite[Section~3]{DDS2025}. 

This behaviour serves as a guiding principle in our setting. 
In fact, using the corresponding property of $B_{\sigma,\theta}$, we first show that the $3$-Brownian shift $T_{\sigma,\theta}$ is also not power-bounded, and hence not similar to a contraction. 
We then prove that an analogous asymptotic phenomenon persists: for every $\sigma>0$ and $\theta\in[0,2\pi)$,
\[
\frac{1}{\|T_{\sigma,e^{i\theta}}\|}\, T_{\sigma,e^{i\theta}} \in C_{00}.
\]

Note that the operator $T_{\sigma,e^{i\theta}}$ admits a decomposition as a perturbation of an isometry. 
Indeed,
\begin{equation*}
T_{\sigma,e^{i\theta}} = T_s + R,
\end{equation*}
where
\[
T_s=
\begin{bmatrix}
M_{z_1} & 0 & 0\\
0 & S & 0\\
0 & 0 & 1
\end{bmatrix}
\]
is an isometry, and
\[
R=
\begin{bmatrix}
0 & \sigma J & 0\\
0 & 0 & \sigma(1\otimes 1)\\
0 & 0 & e^{i\theta}-1
\end{bmatrix}.
\]
This is where it differs from the classical Brownian shift, as
in contrast to the classical case, the perturbation $R$ is no longer of rank one (\cite[P. 87]{DDS2025}). 
Indeed, due to the presence of the operator $J$, the perturbation is of infinite rank. 
Thus, $T_{\sigma,e^{i\theta}}$ may be regarded as an infinite-rank perturbation of the isometry $T_s$.

\subsection{Plan of the paper}

In Section~\ref{sec-2}, we describe invariant subspaces of $T_{\sigma,\theta}$ of the form
$
\mathcal M_0 \oplus \mathcal M_1,
$
where $\mathcal M_0 \subseteq H^2(\mathbb D^2)$ and $\mathcal M_1 \subseteq H^2(\mathbb D)\oplus \mathbb C$. 
This description is then used to identify the unitary operators that arise in the study of unitary equivalence between such invariant subspaces. 
We also show that there exist invariant subspaces of $T_{\sigma,\theta}$ which are not of this form; in particular, this is demonstrated by proving that the vector $(0,0,1)^t$ is not cyclic for $T_{\sigma,\theta}$.

In Section~\ref{sec-3}, we establish several necessary conditions for two invariant subspaces to be unitarily equivalent. 
In Section~\ref{sec-4}, we prove our main result, Theorem~\ref{main}, whose proof relies on Lemma~\ref{cor-dec}, providing structural constraints on the corresponding unitary intertwiner. 

Finally, in Section~\ref{Sec-5}, we show that a normalized $T_{\sigma,\theta}$ belongs to the class $C_{00}$, in analogy with the classical Brownian shift.

\section{Lifted Type Invariant subspaces of $T_{\sigma, \theta}$}\label{sec-2}
In this section, we present a description of invariant subspaces of $T_{\sigma,\theta}$ which we call {\it lifted type invariant subspaces} arising from the invariant subspaces of the classical Brownian shift. These spaces are  of the type
$$\mathcal M:=\begin{bmatrix}
    \mathcal M_0 \\\mathcal M_1
\end{bmatrix}\subseteq \begin{bmatrix}
   H^2(\mathbb D^2) \\ H^2(\mathbb D)\\ \mathbb C 
\end{bmatrix},\qquad \mathcal M_0 \subseteq H^2(\mathbb D^2),\, \, \mathcal M_1\subseteq  \begin{bmatrix} H^2(\mathbb D)\\ \mathbb C \end{bmatrix}.$$
Here $\mathcal M_1\subseteq H^2(\mathbb D)\oplus \mathbb C $ is an invariant subspace of $B_{\sigma, \theta}.$
As noted in the introduction, $\mathcal M_1 $ has two choices, namely Type I and Type II subspaces. 

We now present a description of $\mathcal M$.
To do this, we identify
\[
H^2(\mathbb D^2) \cong H^2\!\big(\mathbb D; H^2(\mathbb D)\big)
\]
with respect to the variable $z_1$. Under this identification, $M_{z_1}$ becomes the unilateral shift on the vector-valued Hardy space $H^2(\mathbb D; H^2(\mathbb D))$.

For an inner function $\varphi \in  H^\infty(\mathbb D),$ we have the following orthogonal decomposition
\[
H^2(\mathbb D)
=
\varphi H^2(\mathbb D)
\oplus
K_{\varphi},
\qquad
K_{\varphi }:= H^2(\mathbb D) \ominus \varphi H^2(\mathbb D).
\]
Consequently,
\[
H^2(\mathbb D^2)
=
H^2\!\big(\mathbb D; \varphi H^2(\mathbb D)\big)
\oplus
H^2\!\big(\mathbb D; K_{\varphi}\big),
\]
and the first summand coincides with $ \varphi_{z_2}\, H^2(\mathbb D^2) $.  For any one variable function $f,$ the notation $f_{z_2}$ is fixed for a function in two variable defined by $f_{z_2}(z_1,z_2):=f(z_2)$.

The following theorem presents a description of invariant subspaces of $T_{\sigma, \theta}$ of the form $\begin{bmatrix}
    \mathcal M_0\\\mathcal M_1
\end{bmatrix}.$

\begin{theorem}
    Let $\sigma>0$ and $ \theta \in [0,2\pi).$ Let $T_{\sigma, \theta}$ be a $3$-Brownian shift and $B_{\sigma, \theta}$ be the Brownian shift. Let $\mathcal M=\begin{bmatrix}
    \mathcal M_0\\\mathcal M_1
\end{bmatrix}$ be an invariant subspace of $T_{\sigma,\theta},$ where $\mathcal M_1$ is a nonzero invariant subspace of $B_{\sigma, \theta}.$ Then $\mathcal M_0$ must be one of the following forms:
    \begin{enumerate}
        \item If $
 \mathcal M_1 = \begin{bmatrix}
    \varphi H^2 \\ 0 \end{bmatrix},
$
where  $\varphi \in H^\infty(\mathbb D),$ is an inner function, then there exists a Hilbert space $\mathcal E_{\mathcal M}$ and an operator-valued inner function
$
\Psi \in H^\infty\!\big(\mathbb D; \mathcal B(\mathcal E_{\mathcal M}, K_\varphi)\big)
$
such that
$$\mathcal M_0
=
 \varphi_{z_2}\, H^2(\mathbb D^2) 
\oplus
\Psi H^2(\mathbb D; \mathcal E_{\mathcal M}).
$$

\item If $
 \mathcal M_1 = \mathbb C
\begin{bmatrix} g \\[2pt] 1 \end{bmatrix}
\oplus \bigl(\varphi H^2 \oplus \{0\}\bigr),
$
where $\varphi\in H^\infty(\mathbb D)$ is an inner function such that $\varphi(e^{i\theta})$
exists, and
\begin{equation*}
g(z)=\sigma\left(\frac{\overline{\varphi(e^{i\theta})}\,\varphi(z)-1}{z-e^{i\theta}}\right),
\qquad z\in\mathbb D,
\end{equation*}
then there exist a Hilbert space $\mathcal E_{\mathcal M}$ and an operator-valued inner function
$
\Psi \in H^\infty\!\big(\mathbb D; \mathcal B(\mathcal E_{\mathcal M}, K_\varphi)\big)
$
such that
\[
\mathcal M_0
=
 \varphi_{z_2}\, H^2(\mathbb D^2) 
\;\oplus\;
g_{z_2}H^2_{z_1}(\mathbb D)
\;\oplus\;
\Psi\,H^2(\mathbb D;\mathcal E_{\mathcal M}),
\]
with $
\Psi(0)\mathcal E_{\mathcal M} \subset g^\perp \subset K_\varphi .
$
    \end{enumerate}
\end{theorem}
\medskip
\begin{proof} Since $T_{\sigma, \theta}(\mathcal M)\subseteq \mathcal M,$ it follows that $\mathcal M_0$ must be invariant under multiplication by $z_1.$ Note that we have only two nontrivial choices for $\mathcal M_1$, which we divide into the following cases.

\noindent\textit{(i)} In this case, we take $\mathcal M_1$ to be 
$$\begin{bmatrix}
    \varphi H^2(\mathbb D)\\0
\end{bmatrix},$$
where  $\varphi \in H^\infty(\mathbb D)$ is an inner function.   From
$$T_{\sigma, \theta}\begin{bmatrix}
    0 \\\varphi f\\ 0
\end{bmatrix} \in \mathcal M=\begin{bmatrix}
    \mathcal M_0 \\\mathcal M_1
\end{bmatrix},$$
we have $\varphi_{z_2}f_{z_2} \in \mathcal M_0$ for every $f \in H^2(\mathbb D).$ Since $\mathcal M_0$ is $M_{z_1}$-invariant, we obtain
$ \varphi_{z_2}\, H^2(\mathbb D^2) \subseteq \mathcal M_0.$

To get a nice description of $\mathcal M_0,$ we will use the Beurling–Lax–Halmos theorem \cite[Theorem~2.1, p. 239]{FF90}. 
Since $ \varphi_{z_2}\, H^2(\mathbb D^2)  \subset \mathcal M_0$, it follows that $\mathcal M_0$ decomposes orthogonally as
\[
\mathcal M_0
=
 \varphi_{z_2}\, H^2(\mathbb D^2) 
\oplus
\mathcal N_0,
\qquad
\mathcal N_0 := \mathcal M_0 \cap H^2\!\big(\mathbb D; K_{\varphi}\big).
\]
The subspace $\mathcal N_0$ is closed and invariant under $M_{z_1}$, now viewed as the shift on $H^2(\mathbb D; K_\varphi)$.
By the Beurling--Lax--Halmos theorem, there exists a Hilbert space $\mathcal E_{\mathcal M}$ and an operator-valued inner function
\[
\Psi \in H^\infty\!\big(\mathbb D; \mathcal B(\mathcal E_{\mathcal M}, K_\varphi)\big)
\]
such that
\[
\mathcal N_0 = \Psi H^2(\mathbb D; \mathcal E_{\mathcal M}).
\]
Therefore,
\[
\mathcal M_0
=
 \varphi_{z_2}\, H^2(\mathbb D^2) 
\oplus
\Psi H^2(\mathbb D; \mathcal E_{\mathcal M}).
\]

\medskip
\noindent\textit{(ii)} In this case, we take the remaining choice for $\mathcal M_1$, i.e,

$$\mathcal M_1=\mathbb C\binom{g}{1}\oplus(\varphi H^2(\mathbb D)\oplus\{0\}),$$
where $\varphi \in  H^2(\mathbb D)$ is inner with $\varphi(e^{i\theta})$ defined and 
\begin{equation*}
g(z)
=
\sigma\,\cdot\frac{\varphi(e^{i\theta})\varphi(z)-1}{z-e^{i\theta}}, \quad z \in \mathbb D.
\end{equation*}
In addition to $ \varphi_{z_2}\, H^2(\mathbb D^2)  \subseteq \mathcal M_0,$ we also have $g_{z_2} \in \mathcal M_0$ in this case. Since $\mathcal M_0$ is $z_1$-invariant, we obtain $z_1^ig_{z_2} \in \mathcal M_0, \, i \geq 0$. Thus $g_{z_2}H_{z_1}^2(\mathbb D) \subseteq \mathcal{M}_0.$ 
Note that we have the orthogonal decomposition
\[
H^2(\mathbb D^2)
=
 \varphi_{z_2}\, H^2(\mathbb D^2) 
\;\oplus\;
H^2(\mathbb D;K_\varphi).
\]
Since the first summand above already lies in $\mathcal M_0$, it follows that
\[
\mathcal M_0
=
 \varphi_{z_2}\, H^2(\mathbb D^2) 
\;\oplus\;
\mathcal N_0,
\qquad
\mathcal N_0 := \mathcal M_0\cap H^2(\mathbb D;K_\varphi),
\]
where $\mathcal N_0$ is a closed $M_{z_1}$-invariant subspace of
$H^2(\mathbb D;K_\varphi)$.
We also know that $g \perp \varphi H^2(\mathbb D),$ which further implies that $g_{z_2} \perp  \varphi_{z_2}\, H^2(\mathbb D^2),$ and thus  $g_{z_2}\in \mathcal N_0$.

\medskip

Consider the wandering subspace of $\mathcal N_0$ with respect to $M_{z_1}$,
\[
W := \mathcal N_0 \ominus z_1 \mathcal N_0 .
\]
By the Wold decomposition for the unilateral shift,
\beq \label{N_0}
\mathcal N_0=\bigoplus_{n\ge 0} z_1^n W .
\eeq
Since $g_{z_2}$ is independent of $z_1$, we necessarily have
$g_{z_2}\perp z_1\mathcal N_0$, hence $g_{z_2}\in W$.
Consequently,
\[
W=\operatorname{span}\{g_{z_2}\}\oplus W',
\qquad
W':=W\ominus \operatorname{span}\{g_{z_2}\}.
\]
From \eqref{N_0}, we have
\[
\mathcal N_0
=
g_{z_2}\,H_{z_1}^2(\mathbb D)
\;\oplus\;
\Big(\bigoplus_{n\ge0} z_1^n W'\Big).
\]
Here $g_{z_2}\,H_{z_1}^2(\mathbb D):=\overline{\operatorname{span}}\{z_1^i\,g_{z_2}; \,i\geq 0\} \subseteq H^2(\mathbb D^2).$

The remaining summand
\[
\mathcal N_0' := \bigoplus_{n\ge0} z_1^n W'
\subset H^2(\mathbb D;K_\varphi)
\]
is again a closed $M_{z_1}$-invariant subspace.
Hence, again by the Beurling--Lax--Halmos theorem, there exist a Hilbert space $\mathcal E_{\mathcal M}$ and an
operator-valued inner function
\[
\Psi \in
H^\infty\!\left(\mathbb D;
\mathcal B(\mathcal E_{\mathcal{M}},K_\varphi)\right)
\]
such that
\[
\mathcal N_0'=\Psi\,H^2(\mathbb D;\mathcal E_{\mathcal{M}}).
\]
Since $W'\perp g$, the wandering space satisfies
$
\Psi(0)\mathcal E_{\mathcal M} \subset g^\perp \subset K_\varphi .
$

Combining the above decompositions, we conclude that every closed
$M_{z_1}$-invariant subspace $N$ that satisfies the stated assumptions
admits the representation
\[
N
=
\varphi_{z_2}\, H^2(\mathbb D^2)
\;\oplus\;
g_{z_2}\, H^2_{z_1}(\mathbb D)
\;\oplus\;
\Psi\,H^2(\mathbb D;\mathcal E_{\mathcal M}),
\]
where $\Psi$ is an inner multiplier taking values in $K_\varphi \ominus \operatorname{span}\{g\}$.
\end{proof}

\begin{remark}
It follows that the minimal lifted Type~I invariant subspace of $T_{\sigma,\theta}$ for which 
$\varphi H^2(\mathbb D)\oplus\{0\},$ is invariant under $B_{\sigma,\theta},$ is given by
\[
\mathcal M^{\mathrm{I}}_\varphi 
= 
 \varphi_{z_2}\, H^2(\mathbb D^2)  
\oplus 
\bigl(\varphi H^2(\mathbb D)\oplus\{0\}\bigr).
\]

Similarly, the minimal lifted Type~II invariant subspace for which 
\[
\mathbb C\begin{bmatrix} g \\ 1 \end{bmatrix}
\oplus
\bigl(\varphi H^2(\mathbb D)\oplus\{0\}\bigr)
\]
is invariant under $B_{\sigma,\theta},$ is
\[
\mathcal M^{\mathrm{II}}_\varphi
:=
N_{\min}
\oplus
\Bigl(
\mathbb C\begin{bmatrix} g \\ 1 \end{bmatrix}
\oplus
\bigl(\varphi H^2(\mathbb D)\oplus\{0\}\bigr)
\Bigr),
\]
where
\[
N_{\min}
=
 \varphi_{z_2}\, H^2(\mathbb D^2) 
\;\oplus\;
g_{z_2}\, H^2_{z_1}(\mathbb D).
\]
\end{remark}

We note that characterizing invariant subspaces for  $T_{\sigma,\theta}$ appears to be a subtle problem. 
Indeed, there are invariant subspaces as cyclic subspaces of the form
\[
\mathscr C_x:=\overline{\operatorname{span}}\{T_{\sigma,\theta}^n x : n \ge 0\},
\qquad x \in H^2(\mathbb D^2)\oplus H^2(\mathbb D)\oplus \mathbb C.
\]
It can be seen from the different choices of $x$ that the invariant subspace  $\mathscr C_x$ does not agree with any of the defined types. For example, consider the vector $x=(0,0,1)^t$. For the Brownian shift $B_{\sigma,\theta}$ 
the vector $(0,1)^t$ is cyclic, that is,
\[
\overline{\operatorname{span}}\{B_{\sigma,\theta}^n(0,1)^t : n \ge 0\}
= H^2(\mathbb D)\oplus \mathbb C .
\]
However, this is not true for $T_{\sigma, \theta},$ if we take the orbit of the vector $(0,0,1)^t$ under $T_{\sigma, \theta}.$  The following proposition illustrates that the structure of invariant subspaces generated by single 
vectors can be quite complicated. For this reason, we restrict our attention to tractable classes of invariant subspaces of $T_{\sigma,\theta}$.

\begin{proposition}[Non-cyclicity of the vector $(0,0,1)^t$]
Let $
T_{\sigma,\theta}$ be a $3$-Brownian shift on
$
\mathcal H
=
H^2(\mathbb D^2)\oplus H^2(\mathbb D)\oplus\mathbb C,
$
where $\sigma> 0, \theta \in [0,2\pi)$. Let $e_3=(0,0,1)^t$. Then $e_3$ is not cyclic for
$T_{\sigma,\theta}$, that is,
\[
\mathcal M:=\overline{\operatorname{span}}\{T_{\sigma,\theta}^n e_3:\ n\ge 0\}\neq
\mathcal H.
\]
\end{proposition}

\begin{proof}
We first show that  for every $n\ge 0$,
\begin{equation}\label{T-n}
T_{\sigma,\theta}^n e_3
=
\Bigl(
\sigma^2 \sum_{k=0}^{n-2} e^{ik\theta}\, h_{n-2-k}(z_1,z_2),
\;
\sigma \sum_{k=0}^{n-1} e^{ik\theta} z^{\,n-1-k},
\;
e^{in\theta}
\Bigr)^t,
\end{equation}
where 
\[
h_m(z_1,z_2)
=
\sum_{\substack{a,b\ge 0\\ a+b=m}} z_1^{a} z_2^{b},
\qquad m\in\mathbb Z,
\]
with the convention that $h_m\equiv 0$ for $m<0$.
 We argue by induction on $n$.

For $n=0$ is trivial and for $n=1$, we have
\[
T_{\sigma,\theta}e_3=(0,\sigma,e^{i\theta})^t,
\]
which agrees with \eqref{T-n}.

Assume that the formula holds for some $n\ge1$, i.e.
\[
T_{\sigma,\theta}^n e_3=(F_n,f_n,e^{in\theta})^t,
\]
with
\[
F_n=\sigma^2\sum_{k=0}^{n-2} e^{ik\theta}\, h_{n-2-k},
\qquad
f_n=\sigma\sum_{k=0}^{n-1} e^{ik\theta} z^{n-1-k}.
\]
Applying $T_{\sigma,\theta}$ yields
\[
T_{\sigma,\theta}^{n+1} e_3
=
\bigl(z_1F_n+\sigma f_n(z_2),\; zf_n+\sigma e^{in\theta},\; e^{i(n+1)\theta}\bigr)^t.
\]
Note that the third coordinate matches with the induction step. For the middle coordinate,
\[
zf_n+\sigma e^{in\theta}
=
\sigma\sum_{k=0}^{n-1} e^{ik\theta} z^{n-k}
+\sigma e^{in\theta}
=
\sigma\sum_{k=0}^{n} e^{ik\theta} z^{n-k}.
\]
For the first coordinate, we compute
\[
\begin{aligned}
z_1F_n+\sigma f_n(z_2)
&=
\sigma^2\sum_{k=0}^{n-2} e^{ik\theta} z_1 h_{n-2-k}
+\sigma^2\sum_{k=0}^{n-1} e^{ik\theta} z_2^{\,n-1-k}.
\end{aligned}
\]
Using the identity $h_{m+1}=z_1h_m+z_2^{m+1}$ for $m\ge0$, we obtain
\[
z_1F_n+\sigma f_n(z_2)
=
\sigma^2\sum_{k=0}^{n-1} e^{ik\theta} h_{n-1-k}.
\]
This establishes the claimed formula for $n+1$ and completes the induction.

Now since a finite linear combination of symmetric polynomials is symmetric, also the pointwise limit of symmetric polynomials is symmetric, it follows that the first coordinate of vectors in $\mathcal M$ is always symmetric in $z_1,z_2.$ Hence, $\mathcal M\neq \mathcal H.$
\end{proof}

\section{Auxillary Results} \label{sec-3}

To investigate the unitary equivalence problem  for $3$-Brownian shifts $T_{\sigma,\theta},$ when restricted to lifted Type I and lifted Type II invariant subspaces, we first compute the operator norm of $T_{\sigma,\theta}$ restricted to these invariant subspaces. Since the operator norm is preserved under unitary equivalence, it provides a useful criterion for distinguishing non-unitarily equivalence.
\begin{lemma}\label{OPOS}
Let $T_{\sigma,\theta}$ be the $3$--Brownian shift acting on
$
\mathcal H
=
H^2(\mathbb D^2)\oplus H^2(\mathbb D)\oplus \mathbb C
$
by
\[
T_{\sigma,\theta}(F,f,\alpha)^t
=
\bigl(z_1F+\sigma f_{z_2},\; zf+\sigma\alpha,\; e^{i\theta}\alpha\bigr)^t,
\qquad \sigma>0.
\]
Let $\mathcal M$ be one of lifted Type I or lifted Type II invariant subspace of $T_{\sigma,\theta}.$
Then
$
\bigl\|T_{\sigma,\theta}\big|_{\mathcal M}\bigr\|
=
\sqrt{1+\sigma^2}.
$
\end{lemma}

\begin{proof}
We first consider the action of $T_{\sigma,\theta}$ on the full space $\mathcal H$.
In $H^2(\mathbb D^2)$, the subspaces $z_1H^2(\mathbb D^2)$ and
$\{g(z_2):g\in H^2(\mathbb D)\}$ are orthogonal, and multiplication by $z_1$
and by $z$ act isometrically on $H^2(\mathbb D^2)$ and $H^2(\mathbb D)$,
respectively. Consequently, for $(F,f,\alpha)^t\in\mathcal H$,
\[
\|T_{\sigma,\theta}(F,f,\alpha)^t\|^2
=
\|F\|^2+(\sigma^2+1)\|f\|^2+(\sigma^2+1)|\alpha|^2
\le
(1+\sigma^2)\|(F,f,\alpha)\|^2.
\]
Thus $\|T_{\sigma,\theta}\|\le\sqrt{1+\sigma^2}.$
Since the equality is attained if we take vectors supported in the second or third
coordinates, we obtain
$
\|T_{\sigma,\theta}\big|_{\mathcal M}\|=\sqrt{1+\sigma^2}.
$
This completes the proof.
\end{proof}
The following lemma provides an immediate necessary condition for the unitary equivalence of invariant subspaces of $3$-Brownian shifts.

\begin{lemma} \label{eq-sig}
Let $\mathcal M_1$ and $M_2$ be lifted Type I or lifted Type II invariant subspaces of $T_{\sigma_1,\theta_1}$ and $T_{\sigma_2,\theta_2}$, respectively.
Assume that $T_{\sigma_1,\theta_1}\big|_{\mathcal \mathcal M_1} \cong T_{\sigma_2,\theta_2}\big|_{\mathcal M_2}.$
Then $\sigma_1 = \sigma_2$.
\end{lemma}

\begin{proof}
The conclusion follows directly from Lemma~\ref{OPOS}.
\end{proof}
We now record the following elementary fact for a unitary equivalent shift-invariant subspaces of a vector-valued Hardy space (see the proof of \cite[Theorem 3.1(i)]{DDS25II}). We add the proof for the sake of completeness.
\begin{proposition}\label{shift-E} 
Let $\psi_i \in H^{\infty}(\mathbb D;B(\mathcal E_i,E))$, $i=1,2$, be operator–valued inner functions. Then $M_z$-invariant subspaces 
$
\psi_1H^2(\mathbb D; \mathcal E_1) $ and $ \psi_2H^2(\mathbb D; \mathcal E_2)
$
are unitarily equivalent if and only if 
$
\dim \mathcal E_1=\dim \mathcal E_2.
$
\end{proposition}

\begin{proof}
Let $N_i:=\psi_iH^2(\mathbb D;\mathcal E_i)\subset H^2(\mathbb D;E)$ for $i=1,2$. Since $\psi_i$ is inner, the multiplication operator 
$M_{\psi_i}:H^2(\mathbb D;\mathcal E_i)\to H^2(\mathbb D;E)$ is an isometry, and $N_i$ is invariant under $M_z$.

First, note that
$
N_i \ominus zN_i
= \psi_i\big(H^2(\mathbb D;\mathcal E_i)\ominus zH^2(\mathbb D;\mathcal E_i)\big).
$
Since $H^2(\mathbb D;\mathcal E_i)\ominus zH^2(\mathbb D;\mathcal E_i)$ consists of constant functions, we obtain
$
N_i \ominus zN_i = \psi_i \mathcal E_i,
$
and therefore
\[
\dim(N_i \ominus zN_i)=\dim \mathcal E_i.
\]

\medskip

\noindent
$(\Rightarrow)$ Suppose that $ N_1$ and $N_2$ are unitarily equivalent with respect to $M_z$. Then there exists a unitary operator 
$U:\mathcal N_1\to N_2$ such that
$
UM_z|_{ N_1}=M_z|_{N_2}U.
$
Hence $U(z N_1)=zN_2$, and consequently
$
U(N_1\ominus z N_1)=N_2\ominus zN_2.
$
Thus
\[
\dim \mathcal E_1
= \dim( N_1\ominus z N_1)
= \dim(N_2\ominus zN_2)
= \dim \mathcal E_2.
\]

\medskip

\noindent
$(\Leftarrow)$ Conversely, assume $\dim \mathcal E_1=\dim \mathcal E_2$. Then there exists a unitary $V:\mathcal E_1\to\mathcal E_2$.  
Using the orthogonal Wold decompositions
\[
N_i=\bigoplus_{n\ge0} z^n (N_i\ominus zN_i)
   =\bigoplus_{n\ge0} z^n \psi_i \mathcal E_i,
\]
define $U:N_1\to N_2$ on the dense algebraic sum by
\[
U\!\left(\sum_{n\ge0} z^n \psi_1 e_n\right)
=
\sum_{n\ge0} z^n \psi_2 (Ve_n),
\qquad e_n\in\mathcal E_1 .
\]
Since the summands $z^n\psi_i\mathcal E_i$ are mutually orthogonal and 
$\psi_i$ acts isometrically on $\mathcal E_i$, the map $U$ is a well-defined unitary operator. 
By construction, $UM_z=M_zU$, so $ N_1$ and $N_2$ are unitarily equivalent with respect to $M_z$.
\end{proof}

\subsection*{Unitary equivalence and compressions}

In the next section, we repeatedly pass from an operator to its compression to a
closed subspace. Since much of our analysis relies on unitary equivalence,
it is useful to record the following straightforward but important permanence
property: unitary equivalence is preserved under compression, provided the
unitary operator maps the relevant subspaces onto one another.

\begin{proposition}\label{prop:compression-intertwining}
Let $\mathcal H_1,\mathcal H_2$ be Hilbert spaces, let $T_1\in\mathcal B(\mathcal H_1)$ and
$T_2\in\mathcal B(\mathcal H_2)$, and let $U:\mathcal H_1\to\mathcal H_2$ be a unitary operator such that
\[
UT_1=T_2U.
\]
Let $\mathcal K_1\subset \mathcal H_1$ and $\mathcal K_2\subset \mathcal H_2$ be closed subspaces with
$U(\mathcal K_1)=\mathcal K_2$. Define the compressions
\[
B_1:=P_{\mathcal K_1}T_1\big|_{\mathcal K_1}\in\mathcal B(\mathcal K_1),
\qquad
B_2:=P_{\mathcal K_2}T_2\big|_{\mathcal K_2}\in\mathcal B(\mathcal K_2).
\]
Then the restriction $U|_{\mathcal K_1}:\mathcal K_1\to\mathcal K_2$ is unitary and intertwines the
compressions, namely
\[
\bigl(U|_{\mathcal K_1}\bigr)\,B_1
=
B_2\,\bigl(U|_{\mathcal K_1}\bigr).
\]
\end{proposition}

\begin{proof}
Since $U(\mathcal K_1)=\mathcal K_2$, one has
\begin{equation}\label{eq:proj-conj}
UP_{\mathcal K_1}=P_{\mathcal K_2}U.
\end{equation}
For $x\in\mathcal K_1$, using \eqref{eq:proj-conj} and $UT_1=T_2U,$ we compute
\[
U B_1 x
=
U P_{\mathcal K_1}T_1x
=
P_{\mathcal K_2}UT_1x
=
P_{\mathcal K_2}T_2Ux
=
B_2Ux,
\]
where the last equality uses $Ux\in\mathcal K_2$. This proves the desired intertwining relation on $\mathcal K_1$.
\end{proof}

\section{Unitary Equivalence}\label{sec-4}

In this section, we prove our main result, Theorem~\ref{main}. 

For convenience, we fix the following notation for this section. Let  $\mathcal M =\begin{bmatrix}
     \mathcal M_0 \\ \mathcal M_1
 \end{bmatrix},\,  \mathcal N =\begin{bmatrix}
     \mathcal N_0 \\ \mathcal N_1
 \end{bmatrix},$ be invariant subspace of $3$-Brownian shit of lifted type.
\begin{enumerate}
\item If  $\mathcal M,  \mathcal N,$  are of lifted Type I,  then
 \[
\begin{aligned}
\mathcal M_0 &= \varphi_{z_2}^{(1)}H^2(\mathbb D^2)
\;\oplus\;
\Psi^{(1)} H^2(\mathbb D;\mathcal E_{\mathcal M}),\\
\mathcal N_0 &= \varphi_{z_2}^{(2)}H^2(\mathbb D^2)
\;\oplus\;
\Psi^{(2)} H^2(\mathbb D;\mathcal E_{\mathcal N}),
\end{aligned}
\]
and
\[
\begin{aligned}
\mathcal M_1 &= 
\varphi^{(1)}H^2(\mathbb D)\oplus\{0\},\\
\mathcal N_1 &= 
\varphi^{(2)}H^2(\mathbb D)\oplus\{0\},
\end{aligned}\]
for some inner function $\varphi^{(j)}\in H^\infty(\mathbb D)$ and  for some operator-valued inner functions $
\Psi^{(j)} \in H^\infty\!\big(\mathbb D;\mathcal B(\mathcal E_{\mathcal M}, K_{\varphi^{(j)}})\big).
$
    \item If $\mathcal M,  \mathcal N,$  are of lifted Type II, then
\[
\begin{aligned}
\mathcal M_0 &= \varphi_{z_2}^{(1)}H^2(\mathbb D^2)
\;\oplus\;
g_{z_2}^{(1)} H^2_{z_1}(\mathbb D)
\;\oplus\;
\Psi^{(1)} H^2(\mathbb D;\mathcal E_{\mathcal M}),\\
\mathcal N_0 &= \varphi_{z_2}^{(2)}H^2(\mathbb D^2)
\;\oplus\;
g_{z_2}^{(2)} H^2_{z_1}(\mathbb D)
\;\oplus\;
\Psi^{(2)} H^2(\mathbb D;\mathcal E_{\mathcal N}),
\end{aligned}
\]
and
\[
\begin{aligned}
\mathcal M_1 &= \mathbb C
\begin{bmatrix}
g^{(1)}\\
1
\end{bmatrix}
\oplus
\bigl(\varphi^{(1)}H^2(\mathbb D)\oplus\{0\}\bigr),\\
\mathcal N_1 &= \mathbb C
\begin{bmatrix}
g^{(2)}\\
1
\end{bmatrix}
\oplus
\bigl(\varphi^{(2)}H^2(\mathbb D)\oplus\{0\}\bigr),
\end{aligned}
\]
with
\[
g^{(j)}
=
\sigma_j\cdot
\frac{\overline{\varphi^{(j)}(e^{i\theta_j})}\,\varphi^{(j)}-1}{z-e^{i\theta_j}},
\qquad j=1,2,
\]
for some inner function $\varphi^{(j)}\in H^\infty(\mathbb D)$ and  for some operator-valued inner functions $$
\Psi^{(1)} \in H^\infty\!\big(\mathbb D;\mathcal B(\mathcal E_{\mathcal M}, K_{\varphi^{(1)}})\big),\, \text{with} \,\,\,
\Psi(0)\mathcal E_{\mathcal M} \subset {g^{(1)}}^\perp \subset K_{\varphi^{(1)}},
$$
 $$
\Psi^{(2)} \in H^\infty\!\big(\mathbb D;\mathcal B(\mathcal E_{\mathcal M}, K_{\varphi^{(2)}})\big),\, \text{with }\,\, \,
\Psi(0)\mathcal E_{\mathcal N} \subset {g^{(2)}}^\perp \subset K_{\varphi^{(2)}}.
$$

\end{enumerate}

In the following lemma, we provide necessary conditions for the existence of a unitary intertwiner between lifted-type invariant subspaces.
In particular, such an operator must preserve the coordinate decomposition and act diagonally, inducing a shift-commuting unitary on the Hardy component.

\begin{lemma}\label{cor-dec}
 Let  $\mathcal M =\begin{bmatrix}
     \mathcal M_0 \\ \mathcal M_1
 \end{bmatrix},\,  \mathcal N =\begin{bmatrix}
     \mathcal N_0 \\ \mathcal N_1
 \end{bmatrix}$ be lifted type invariant subspaces of $T_{\sigma,\theta_1}$ and $T_{\sigma,\theta_2}$,  respectively.
Suppose that there exists a unitary operator
$U: \mathcal  M\to  \mathcal N$ such that
$$
UT_{\sigma,\theta_1}|_{ \mathcal  M}=T_{\sigma,\theta_2}|_{ \mathcal  N}U.
$$
Then the following assertions hold.
\begin{enumerate}
\item The unitary operator $U$ preserves the coordinate decomposition and hence admits the block representation
\[
U=
\begin{bmatrix}
U_1 & 0\\
0 & U_2
\end{bmatrix},
\]
where $U_1:\mathcal M_0\to \mathcal N_0$ and $U_2:\mathcal M_1\to \mathcal N_1$ are unitary operators.

\item The operator $U_1$ commutes with  $M_{z_1},$ i.e. ,
$
U_1(z_1F)=z_1U_1(F),
$
for all $F\in \mathcal M_0.$

\item $\mathcal  M, $ and $ \mathcal N$ must be of the lifted  type. The unitary operator $U_2$ maps $\varphi^{(1)}H^2(\mathbb D)\oplus \{0\}$ onto $\varphi^{(2)} H^2(\mathbb D) \oplus \{0\}$ via
\[
U_2
\begin{bmatrix}
\varphi^{(1)} f\\
0
\end{bmatrix}
=
\begin{bmatrix}
\lambda\, \varphi^{(2)} f\\
0
\end{bmatrix},
\qquad f \in H^2(\mathbb D),
\]
for some $\lambda \in \mathbb T$. 
Accordingly, we shall identify $U_2$ with a unitary map from $\varphi^{(1)}H^2(\mathbb D)$ onto $\varphi^{(2)}H^2(\mathbb D)$. Moreover, in the Type II case, $U_{2}$ maps $\operatorname{span} \begin{bmatrix}
    g^{(1)}\\ 1
\end{bmatrix}$ onto $\operatorname{span}\begin{bmatrix}
    g^{(2)}\\ 1
\end{bmatrix}$, i.e.,
$$ U_2\begin{bmatrix}
    g^{(1)}\\ 1
\end{bmatrix}=\beta\begin{bmatrix}
    g^{(2)}\\ 1
\end{bmatrix}$$
for some $\beta \in \mathbb C\setminus \{0\}.$

\item Considering the isometric
embedding $J:H^2(\mathbb D)\to H^2(\mathbb D^2),$ defined by $J(h)(z_1,z_2)=h(z_2),$ we have 
\[
U_1\bigl(J(G)\bigr)=J\bigl(U_2
    G \bigr),
\qquad G\in \varphi^{(1)} H^2(\mathbb D),
\]  and $U_1(Jg^{(1)})=\beta Jg^{(2)}$ for some $\beta \in \mathbb C\setminus \{0\}.$
\item $\dim \mathcal{ E_{M}} $ and $\dim \mathcal{ E_{N}}$ are equal.
\end{enumerate}
\end{lemma}

\begin{proof}
$(i)$ Let $F\in \mathcal M_0 $ and  
\begin{equation}\label{im-U}
U\begin{bmatrix}
    F\\0\\0
\end{bmatrix}= \begin{bmatrix}
    G\\g\\c 
\end{bmatrix} \in   \begin{bmatrix}
  \mathcal N_0\\ \mathcal N_1 
\end{bmatrix},
\end{equation}
where $G \in \mathcal N_0$ and $\begin{bmatrix}
    g\\c 
\end{bmatrix} \in \mathcal N_1.$ Since $UT_{\sigma, \theta_1}|_{\mathcal M}=T_{\sigma, \theta_2}|_{\mathcal N}U,$ we have
 $$U\begin{bmatrix}
    z_1F\\0\\0
\end{bmatrix}=\begin{bmatrix}
    z_1G+\sigma g_{z_2}\\zg+c\sigma\\c e^{i\theta_2}
\end{bmatrix}.$$
Computing the norm of both sides, we get 
\beq\label{norm-F}
\|F\|^2=\|G\|^2+(\sigma^2+1)(\|g\|^2+|c|^2).
\eeq
By \eqref{im-U}, we have $\|F\|^2=\|G\|^2+\|g\|^2+|c|^2.$ This together with \eqref{norm-F} yields that $g=0$ and $c=0.$
We obtain
$U(F,0,0)^{t}=(U_1F,0,0)^{t}$ for some isometry $U_1:\mathcal M_0 \rightarrow \mathcal N_0$.
Applying the same argument to $U^{-1}$ shows that $U_1$ is onto. Hence, $U_1$ is unitary. A similar argument works in the case $\begin{bmatrix}
    0\\ \mathcal M_1
\end{bmatrix}.$ Thus we have the required unitary operators $U_1:\mathcal M_0 \rightarrow \mathcal N_0$ and $U_2: \mathcal M_1\rightarrow \mathcal N_1$.

\noindent$(ii)$  The intertwining relation $T_{\sigma, \theta_1}|_{\mathcal M}U=UT_{\sigma,\theta_2}|_{\mathcal N}$ applied to $(F,0,0)^{t}$ gives
\[
U_1(z_1F)=z_1U_1(F), \qquad F \in \mathcal M_0.
\]

\noindent$(iii)$ Note that if a vector $$\begin{bmatrix}
    0\\G\\0
\end{bmatrix} \in \begin{bmatrix}
    \mathcal M_0\\ \mathcal M_1
\end{bmatrix},$$
then $G$ must be in $\varphi^{(1)} H^2(\mathbb D).$
Let
\begin{equation} \label{Umap}U\begin{bmatrix}
   0\\ \varphi^{(1)} f \\0 
\end{bmatrix}=\begin{bmatrix}
    0\\ g\\c
\end{bmatrix}, \quad f \in H^2 (\mathbb D), \,\begin{bmatrix}
g\\c
\end{bmatrix} 
\in \mathcal M_1.
\end{equation}
Then
$$T_{\sigma,\theta_2}U\begin{bmatrix}
   0\\ \varphi^{(1)} f \\0 
\end{bmatrix}=T_{\sigma,\theta_2}\begin{bmatrix}
    0\\ g\\c
\end{bmatrix}.$$
By the intertwining relation $UT_{\sigma,\theta_1}|_{\mathcal M}=T_{\sigma,\theta_2}|_{\mathcal{N}}U$, we have
$$UT_{\sigma,\theta_1}\begin{bmatrix}
   0\\ \varphi^{(1)} f \\0 
\end{bmatrix}=U\begin{bmatrix}
   \sigma \varphi_{1_{z_2}}f_{z_2}\\ z\varphi _1f \\0 
\end{bmatrix}=\begin{bmatrix}
    \sigma g_{z_2}\\ zg+\sigma c\\ce^{i\theta_2}
\end{bmatrix}=T_{\sigma,\theta_2}\begin{bmatrix}
    0\\ g\\c
\end{bmatrix}.$$
But this yields that $U_{1} (\sigma \varphi_{1_{z_2}}f_{z_2})=\sigma g_{z_2}.$ Thus $\|f\|=\|g\|.$ From \eqref{Umap}, we obtain $c=0$ and  $U_2$  maps $\varphi^{(1)} H^2(\mathbb D) \oplus \{0\} $ onto $ \varphi^{(2)} H^2(\mathbb D) \oplus \{0\}.$  Furthermore, since $U_2$ is unitary, $U_2$ must map the orthogonal complement of $\varphi^{(1)}H^2(\mathbb D)\oplus \{0\}$ in $\mathcal M_1$ to the orthogonal complement of $\varphi^{(2)}H^2(\mathbb D) \oplus \{0\}$ in $\mathcal N_1.$  Thus, we can assume that  in the Type II case, $U_2\begin{bmatrix}
    g^{(1)} \\1
\end{bmatrix}=\beta\begin{bmatrix}
    g^{(2)} \\1
\end{bmatrix}$
 for some $\beta \in \mathbb C$. This shows that $\mathcal M,$ $\mathcal N$ must be of the same type.

By comparing the second coordinate in
\[
U T_{\sigma,\theta_1}
\begin{bmatrix}
0\\
\varphi^{(1)}\\
0
\end{bmatrix}
=
T_{\sigma,\theta_2}\big|_{\mathcal N}
\,U
\begin{bmatrix}
0\\
\varphi^{(1)}\\
0
\end{bmatrix},
\]
we obtain
\[
U_2( z\varphi^{(1)})
=z\,U_2
    (\varphi^{(1)} ).
\]
It follows that
$$U_2(\varphi^{(1)}H^2(\mathbb D) )= U_2
    (\varphi^{(1)})H^2(\mathbb D)
.$$ Since $U_2(\varphi^{(1)} H^2(\mathbb D))=\varphi^{(2)} H^2(\mathbb D),$ we obtain $U_2(\varphi^{(1)})=\lambda \varphi^{(2)}$ for some $|\lambda|=1.$

\noindent$(iv)$ This is straightforward from the relation $UT_{\sigma, \theta_1}|_{\mathcal M}=T_{\sigma, \theta_2}|_{\mathcal N}U$ applied on the vector $(0,G,0)^t,\, G \in \varphi^{(1)}H^2(\mathbb D).$

\noindent$(v)$ Assume that $\mathcal M$ and $\mathcal N$ are of Type~I.
By part~(i), the unitary $U_1$ maps $\mathcal M_0$ onto $\mathcal N_0$. 
Moreover, by parts~(iii) and~(iv), $U_1$ maps 
$\varphi_{z_2}^{(1)} H^2(\mathbb D^2)$ onto 
$\varphi_{z_2}^{(2)} H^2(\mathbb D^2)$. 
Consequently, $U_1$ must map
\[
\Psi^{(1)} H^2(\mathbb D; \mathcal E_{\mathcal M})
\quad \text{onto} \quad
\Psi^{(2)} H^2(\mathbb D; \mathcal E_{\mathcal N}).
\]
By part (ii), we note that $U_1(z_1 F)=z_1 U_1(F)$ for $F\in \mathcal M_0.$ Since $U_1$ intertwines $M_{z_1}$, the subspaces 
\[
\Psi^{(1)} H^2(\mathbb D;\mathcal E_{\mathcal M})
\quad \text{and} \quad
\Psi^{(2)} H^2(\mathbb D;\mathcal E_{\mathcal N})
\]
are unitarily equivalent $M_{z_1}$-invariant subspaces. 
Hence, by the Proposition~\ref{shift-E} theorem,
\[
\dim \mathcal E_{\mathcal M}=\dim \mathcal E_{\mathcal N}.
\]
An analogous argument yields the same conclusion in the Type~II case.
\end{proof}
\begin{remark}
From part~(iv) of Lemma~\ref{cor-dec}, we observe that certain information carried by $U_2$ is reflected in $U_1$. 
This motivated a natural strategy to decompose the first coordinate into several components, according to the structural relationship between $U_1$ and $U_2$.

\end{remark}

Let us now proceed to prove our main theorem, which we restate here for convenience.
\begin{theorem}
Let $\theta_1,\theta_2\in[0,2\pi)$ and $\sigma_1,\sigma_2>0$. 
Let
\[
\mathcal M=
\begin{bmatrix}
\mathcal M_0\\
\mathcal M_1
\end{bmatrix},
\qquad
\mathcal N=
\begin{bmatrix}
\mathcal N_0\\
\mathcal N_1
\end{bmatrix}
\]
be nonzero closed invariant subspaces of lifted Type~I or Type~II for the $3$-Brownian shifts 
$T_{\sigma_1,\theta_1}$ and $T_{\sigma_2,\theta_2}$, respectively. Then
\[
T_{\sigma_1,\theta_1}\big|_{\mathcal M}
\cong
T_{\sigma_2,\theta_2}\big|_{\mathcal N}
\]
if and only if $\sigma_1=\sigma_2, \, \dim \mathcal E_{\mathcal M}=\dim \mathcal E_{\mathcal N}$ and one of the following holds:

\begin{enumerate}
\item[(i)] Both $\mathcal M$ and $\mathcal N$ are of lifted Type~I.
\medskip
\item[(ii)] Both $\mathcal M$ and $\mathcal N$ are of lifted Type~II,  $
\theta_1=\theta_2,$ and $\|g^{(1)}\|=\|g^{(2)}\|.
$
\end{enumerate}
\end{theorem}

\begin{proof}
    To show the necessity part, assume that 
    \[
T_{\sigma_1,\theta_1}\big|_{\mathcal M}\cong
T_{\sigma_2,\theta_2}\big|_{\mathcal N},
\]
and let $U:\mathcal M \rightarrow \mathcal N$ be a unitary operator  such that $UT_{\sigma_1,\theta_1}\big|_{\mathcal M}=
T_{\sigma_2,\theta_2}\big|_{\mathcal N}U.$ From Lemma~\ref{eq-sig}, $\sigma_1=\sigma_2$. 
By Lemma~\ref{cor-dec}(iii), the invariant subspaces $\mathcal{M}$ and $\mathcal N$
are of the same type and by Lemma~~\ref{cor-dec}(v), $\dim \mathcal E_{\mathcal M}=\dim \mathcal E_{\mathcal N}.$ 

Furthermore, by Lemma~\ref{cor-dec}(i), the intertwining unitary $U$ decomposes as
\[
U = U_1 \oplus U_2
\quad \text{from } \mathcal M_0 \oplus \mathcal M_1 \text{ onto } \mathcal N_0 \oplus \mathcal N_1.
\]
In addition, Proposition~\ref{prop:compression-intertwining} yields
\[
U_2\, B_{\sigma,\theta_1}\big|_{\mathcal M_1}
=
B_{\sigma,\theta_2}\big|_{\mathcal N_1}\, U_2.
\]
It now follows from the unitary equivalence theory for Brownian shifts 
\cite[Theorem~2.1]{DDS2025} that, if $\mathcal M_1$ and $\mathcal N_1$ are of Type~II, then
\[
\theta_1=\theta_2
\quad \text{and} \quad
\|g^{(1)}\|=\|g^{(2)}\|.
\]

Now to see the sufficiency part, assume that $\sigma:=\sigma_1=\sigma_2.$ We divide the proof into  the following two cases according to the type of $\mathcal M, \mathcal{N}$:

\noindent \textit{(i) $\mathcal M, \mathcal N$ are of lifted Type I.} 
Since $\dim \mathcal E_{\mathcal{M}}=\dim \mathcal E_{\mathcal{N}},$ take a unitary $V: \mathcal E_{\mathcal M} \rightarrow \mathcal E_{\mathcal N}$ and define a  unitary operator 
$V_1:\psi^{(1)} H^2(\mathbb D; \mathcal E_{\mathcal M}) \rightarrow \psi^{(2)} H^2(\mathbb D; \mathcal E_{\mathcal N})$ by
$$V_1\left(\sum_{i=0}^{\infty} z^i  \psi^{(1)}v_i\right)= \sum_{i=0}^{\infty} z^i \psi^{(2)} V  v_i, \quad v_i \in \mathcal{E}_{\mathcal M}.$$

\noindent Clearly $V_1$ commutes with $M_z.$ Now define a natural choice of  unitary operators $$U_1:\varphi^{(1)}_{z_2}H^2(\mathbb D^2)\rightarrow \varphi^{(2)}_{z_2}H^2(\mathbb D),\,\,\, U_2:
\varphi^{(1)}_{z_2}H^2(\mathbb D^2)\oplus  \{0\}
\rightarrow \varphi^{(2)}_{z_2}H^2(\mathbb D^2)\oplus  \{0\} $$ by $$ U_1(\varphi_{z_2}^{(1)}G)=\varphi_{z_2}^{(2)}G, \qquad U_2\begin{bmatrix}
    \varphi^{(1)}f\\ 0
\end{bmatrix}= \begin{bmatrix} \varphi^{(2)}f\\ 0
\end{bmatrix}, \quad G\in H^2(\mathbb D^2), \,f \in H^2(\mathbb D).$$

In this case, the unitary $U: \mathcal M\rightarrow \mathcal N,$ is defined by
$$U\begin{bmatrix}
    F_1+F_2\\f\\0
\end{bmatrix}= \begin{bmatrix}
    U_1F_1+VF_2\\U_2\begin{bmatrix}
        f\\0
    \end{bmatrix}
\end{bmatrix},
$$ 
where $F_1 \in \varphi^{(1)}_{z_2}H^2(\mathbb D^2), \,F_2 \in \psi^{(1)} H^2(\mathbb D; \mathcal E_{\mathcal M}), \,f \in \varphi^{(1)} H^2(\mathbb D^2).$ Now it is easy to verify that $UT_{\sigma, \theta_1}|_\mathcal M = T_{\sigma, \theta_2}|_\mathcal NU$.
Indeed, it suffices to verify the intertwining on vectors of the form

\begin{equation*}
\begin{bmatrix}F \\ 0\\ 0\end{bmatrix}, \, F \in M_0\\, \, \text{ and }
\begin{bmatrix}0\\ f\\ 0\end{bmatrix}, \,f\in\varphi^{(1)}H^2(\mathbb D).
\end{equation*}
It is easy to verify the intertwining relation for the first vector using Lemma~\ref{cor-dec}(ii). To verify this for the second vector,  write $f=\varphi^{(1)}h$, $h\in H^2(\mathbb D)$. Then
\[
\begin{aligned}
UT_{\sigma,\theta_1}\!\begin{bmatrix}0\\ f\\ 0\end{bmatrix}
&=U\!\begin{bmatrix}\sigma Jf\\ Sf\\ 0\end{bmatrix}
 =\begin{bmatrix}\sigma U_1Jf\\ U_2\begin{bmatrix}
        Sf\\0
    \end{bmatrix}\end{bmatrix} \\
&=\begin{bmatrix}\sigma U_1(\varphi_{z_2}^{(1)}h)\\ U_2B_{\sigma, \theta_1}\begin{bmatrix}
        f\\0
    \end{bmatrix}\end{bmatrix}
 =\begin{bmatrix}\sigma \varphi_{z_2}^{(2)}h\\ B_{\sigma, \theta_2}U_2\begin{bmatrix}
        f\\0 \end{bmatrix}\end{bmatrix} \\
&=\begin{bmatrix}\sigma J\varphi^{(2)}h\\ B_{\sigma, \theta_2}U_2\begin{bmatrix}
        f\\0 \end{bmatrix}\end{bmatrix}
 =T_{\sigma,\theta_2}U\!\begin{bmatrix}0\\ f\\ 0\end{bmatrix}.
\end{aligned}
\]

\noindent \textit{(ii) $\mathcal M, \mathcal N$ are of lifted Type II.}
The components
\[
\varphi_{z_2}^{(i)}H^2(\mathbb D^2), \Psi^{(i)}H^2(\mathbb D;\mathcal E_{\mathcal M})
\quad\text{and}\quad
\varphi^{(i)}H^2(\mathbb D)
\]
are handled exactly as in part (i). Thus, it remains only to define unitary maps on the additional summands.

Define a linear map
\[
\widetilde U:
\operatorname{span}
\begin{bmatrix}
g^{(1)}\\
1
\end{bmatrix}
\longrightarrow
\operatorname{span}
\begin{bmatrix}
g^{(2)}\\
1
\end{bmatrix}
\]
by
\[
\widetilde U
\begin{bmatrix}
g^{(1)}\\
1
\end{bmatrix}
=
\begin{bmatrix}
g^{(2)}\\
1
\end{bmatrix}.
\]
Since $\|g^{(1)}\|=\|g^{(2)}\|$, this map is unitary.

Next, define
\[
\widetilde V:
g_{z_2}^{(1)}H^2_{z_1}(\mathbb D)
\longrightarrow
g_{z_2}^{(2)}H^2_{z_1}(\mathbb D)
\]
by
\[
\widetilde V(z_1^k g_{z_2}^{(1)})=z_1^k g_{z_2}^{(2)}, \qquad k\ge0.
\]
Now extend it by linearity and continuity to all of \( g_{z_2}^{(1)}H^2_{z_1}(\mathbb D) \). Then $\widetilde V$ is unitary and commutes with $M_{z_1}$.

Combining these maps with the unitary operators constructed in part (i), we obtain a unitary
$
U:\mathcal M\to\mathcal N
$
such that
\[
UT_{\sigma,\theta_1}\big|_{\mathcal M}
=
T_{\sigma,\theta_2}\big|_{\mathcal N}U.
\]
To verify this it suffices to check the intertwining relation on the vector
$$
\begin{bmatrix}
g^{(1)}f \\
0\\
0
\end{bmatrix}, \quad f \in H^2_{z_1}(\mathbb D) \text{ and}\quad 
\begin{bmatrix}
0 \\
g_{z_2}^{(1)}\\
1
\end{bmatrix}.
$$
On the first vector, it is easy to verify, so we verify on the second vector. By definition,
\[
T_{\sigma,\theta_1}
\begin{bmatrix}
0\\[1mm]
g^{(1)}\\[1mm]
1
\end{bmatrix}
=
\begin{bmatrix}
\sigma Jg^{(1)}\\[1mm]
Sg^{(1)}+\sigma\\[1mm]
e^{i\theta_1}
\end{bmatrix}.
\]
Applying \(U\), we obtain
\[
UT_{\sigma,\theta_1}
\begin{bmatrix}
0\\[1mm]
g^{(1)}\\[1mm]
1
\end{bmatrix}
=
\begin{bmatrix}
\sigma \widetilde V(Jg^{(1)})\\[1mm]
\widetilde U\!\begin{bmatrix}Sg^{(1)}+\sigma\\[1mm] e^{i\theta_1}\end{bmatrix}
\end{bmatrix}.
\]
Now, by the definition of \(\widetilde V\),
\[
\widetilde V(Jg^{(1)})=Jg^{(2)},
\]
and since $\widetilde U$ intertwines the restricted Brownian shifts on $\mathcal M_1$ and $\mathcal N_1$, i.e,
\[
\widetilde U\, B_{\sigma,\theta_1}\big|_{\mathcal M_1}
=
B_{\sigma,\theta_2}\big|_{\mathcal N_1}\, \widetilde U,
\]
we obtain
\[
\widetilde U
\begin{bmatrix}
Sg^{(1)}+\sigma\\[1mm]
e^{i\theta_1}
\end{bmatrix}
=
\widetilde U\, B_{\sigma,\theta_1}
\begin{bmatrix}
g^{(1)}\\[1mm]
1
\end{bmatrix}
=
B_{\sigma,\theta_2}\,
\widetilde U
\begin{bmatrix}
g^{(1)}\\[1mm]
1
\end{bmatrix}
=
\begin{bmatrix}
Sg^{(2)}+\sigma\\[1mm]
e^{i\theta_2}
\end{bmatrix}.
\] 
 Therefore,
\[
UT_{\sigma,\theta_1}
\begin{bmatrix}
0\\[1mm]
g^{(1)}\\[1mm]
1
\end{bmatrix}
=
\begin{bmatrix}
\sigma Jg^{(2)}\\[1mm]
Sg^{(2)}+\sigma\\[1mm]
e^{i\theta_2}
\end{bmatrix}
=
T_{\sigma,\theta_2}
\begin{bmatrix}
0\\[1mm]
g^{(2)}\\[1mm]
1
\end{bmatrix}
=
T_{\sigma,\theta_2}U
\begin{bmatrix}
0\\[1mm]
g^{(1)}\\[1mm]
1
\end{bmatrix}.
\]
Hence,
\[
UT_{\sigma,\theta_1}
\begin{bmatrix}
0\\[1mm]
g^{(1)}\\[1mm]
1
\end{bmatrix}
=
T_{\sigma,\theta_2}U
\begin{bmatrix}
0\\[1mm]
g^{(1)}\\[1mm]
1
\end{bmatrix}.
\]
This completes the proof in the lifted Type II case.
\end{proof}

\section{$\frac{1}{\|T_{\sigma, \theta}\|}T_{\sigma, \theta}$ belongs to the class $ C_{00}$} \label{Sec-5}

 In this section, we show that a $3$-Brownian shift $T_{\sigma, \theta} $ does not belong to $ C_{00}$ by showing that it is not power bounded. Then we prove that the normalized $3$-Brownian shift $\frac{1}{\sqrt{1+\sigma^2}}T_{\sigma, \theta} $ belongs to the class $ C_{00}.$  The results and techniques are analogous to the results obtained for Brownian shifts in \cite[Section~3]{DDS2025}.

\begin{proposition}
For $\sigma>0, \theta \in [0,2\pi),$ the operator $T_{\sigma,\theta}$ is not power bounded.
\end{proposition}
\begin{proof}
This follows from
\begin{equation*}
\|T^n_{\sigma, \theta}\|^2\geq \|T^n_{\sigma, \theta}(0,0,1)\|^2\geq \|B^n_{\sigma, \theta}(0,1)\|^2\geq 1+n \sigma^2, \quad n \geq 0. \qedhere
\end{equation*}
\end{proof}

The following is the main result of this section.

\begin{theorem}
For $\sigma>0, \theta \in [0,2\pi),$ let $
\widetilde T_{\sigma, \theta}
:=
\frac{1}{\sqrt{1+\sigma^2}}\,T_{\sigma,\theta}.
$ Then $$\widetilde T_{\sigma,\theta} \in C_{00}.$$
\end{theorem}

\begin{proof}
Let $\widetilde T\equiv\widetilde T_{\sigma, \theta}.$ We will show that $\widetilde T^n\to0$ and $\widetilde T^{*n}\to0$ in the strong operator topology. To show $\widetilde T^n\to0$ in SOT, take
 $u=(F,f,\alpha)^t\in\mathcal H$. We have
\[
\|T_{\sigma,\theta}^n u\|
\le
\|T_{\sigma,\theta}^n(F,0,0)^t\|
+\|T_{\sigma,\theta}^n(0,f,0)^t\|
+\|T_{\sigma,\theta}^n(0,0,\alpha)^t\|,\quad n \in \mathbb Z_+.
\]
The first term equals $\|F\|,$ since $M_{z_1}$ is an isometry.
A direct computation as above shows for $ n \in \mathbb Z_+,$
\[
\|T_{\sigma,\theta}^n(0,f,0)^t\|^2
\le (1+\sigma^2 n)\|f\|^2,
\qquad
\|T_{\sigma,\theta}^n(0,0,\alpha)^t\|^2
\le C_\sigma(1+n^2)|\alpha|^2,
\]
for some constant $C_\sigma>0$.
Consequently,
\[
\|\widetilde T^n u\|
\le
\frac{\|F\|+\sqrt{1+\sigma^2 n}\,\|f\|
+\sqrt{C_\sigma(1+n^2)}\,|\alpha|}
{(1+\sigma^2)^{n/2}}
\;{\longrightarrow}\;0  \quad (n \rightarrow \infty),
\]
since the exponential decay dominates the polynomial growth. Hence, 
$\widetilde T^n\to0$ in SOT.

\smallskip

We now prove that $\widetilde T^{*n}\to 0$ in SOT.  
Since $\|\widetilde T\|\le 1$, it suffices to verify the convergence on a dense subset of
$\mathcal H=H^2(\mathbb D^2)\oplus H^2(\mathbb D)\oplus \mathbb C$.
Recall that
\[
T_{\sigma, \theta}=
\begin{bmatrix}
M_{z_1} & \sigma J & 0\\
0 & S & \sigma(1\otimes 1)\\
0 & 0 & e^{i\theta}
\end{bmatrix},
\qquad
\widetilde T_{\sigma, \theta} = \frac{1}{\sqrt{1+\sigma^2}}\,T_{\sigma, \theta},
\]
where $J:H^2(\mathbb D)\to H^2(\mathbb D^2)$ is given by $(Jf)(z_1,z_2)=f(z_2)$.
The adjoint of $T$ has the block form
\[
T^*=
\begin{bmatrix}
M_{z_1}^* & 0 & 0\\
\sigma J^* & S^* & 0\\
0 & \sigma(1\otimes 1)^* & e^{-i\theta}
\end{bmatrix}.
\]
Note that
\[
J^*(z_1^m z_2^k)=
\begin{cases}
z^k, & m=0,\\
0, & m\ge 1,
\end{cases}
\qquad
(1\otimes 1)^*(f)=\langle f,1\rangle,
\]
and that $M_{z_1}^*$ and $S^*$ are the backward shifts on $H^2(\mathbb D^2)$ and
$H^2(\mathbb D)$, respectively.

Since finite linear combinations of monomials are dense, it suffices to check strong
convergence on the following orthonormal basis vectors:
\[
(z_1^m z_2^k,0,0)^t,\qquad (0,z^k,0)^t,\qquad (0,0,1)^t.
\]

\medskip

\noindent\textbf{Case 1.} $x=(0,0,1)^t$.  
Then $T^{*n}x=(0,0,e^{-in\theta})^t$, and hence
\[
\|\widetilde T^{*n}x\| = (1+\sigma^2)^{-n/2} {\longrightarrow}0 \quad (n \rightarrow \infty).
\]

\medskip

\noindent\textbf{Case 2.} $x=(0,z^k,0)^t$, $k\ge 0$.  
If $k\ge 1$, then
\[
T^*(0,z^k,0)^t=(0,z^{k-1},0)^t,
\]
while
\[
T^*(0,1,0)^t=(0,0,\sigma)^t.
\]
Consequently, after at most $k+1$ iterations, the vector $(0,z^k,0)$ enters the scalar component
$\mathbb C$, and thereafter remains bounded in norm by $\max\{1,\sigma\}$. Therefore,
for all $n$,
\[
\|T^{*n}(0,z^k,0)^t\|\le C_k
\]
for some constant $C_k$ independent of $n$. It follows that
\[
\|\widetilde T^{*n}(0,z^k,0)^t\|
\le C_k (1+\sigma^2)^{-n/2} {\longrightarrow} 0  \quad (n \rightarrow \infty) .
\]

\medskip

\noindent\textbf{Case 3.} $x=(z_1^m z_2^k,0,0)^t$, $m,k\ge 0$.  
If $m\ge 1$, then
\[
T^*(z_1^m z_2^k,0,0)^t=(z_1^{m-1}z_2^k,0,0)^t,
\]
and hence
\[
T^{*m}(z_1^m z_2^k,0,0)^t=(z_2^k,0,0)^t.
\]
Applying $T^*$ once more yields
\[
T^*(z_2^k,0,0)^t=(0,\sigma z^k,0)^t.
\]
As in Case~2, after finitely many additional iterations the vector enters the scalar
component and remains uniformly bounded. Thus there exists a constant $C_{m,k}>0$
such that
\[
\|T^{*n}(z_1^m z_2^k,0,0)^t\|\le C_{m,k}
\quad \text{for all } n\ge 1.
\]
Consequently,
\[
\|\widetilde T^{*n}(z_1^m z_2^k,0,0)^t\|
\le C_{m,k}(1+\sigma^2)^{-n/2}  {\longrightarrow} 0  \quad (n \rightarrow \infty) .
\]

\medskip

Combining the above cases, we conclude that $\widetilde T^{*n}x\to 0$ for all basis
vectors $x$. Since $\sup_n\|\widetilde T^{*n}\|\le 1$ and these vectors span a dense
subspace of $\mathcal H$, it follows that
$
\widetilde T^{*n} \xrightarrow{\mathrm{SOT}} 0.
$ This completes the proof.
\end{proof}

\section*{Acknowledgements}

The author thanks A. Zalar for several helpful comments. The author was supported by the ARIS (Slovenian Research and Innovation Agency) research core funding No.\ P1-0288 and grant No.\ J1-60011.

\end{document}